\documentclass[12pt,a4paper]{article}

\usepackage{amsfonts}
\usepackage{amssymb}
\usepackage{amsmath}

\usepackage{xcolor}
\usepackage[linkcolor=blue,colorlinks=true,urlcolor=blue,citecolor=blue]{hyperref}
\usepackage{cite}

\newtheorem{theorem}{Theorem}

\newtheorem{corollary}[theorem]{Corollary}
\newtheorem{definition}[theorem]{Definition}
\newtheorem{example}[theorem]{Example}
\newtheorem{lemma}[theorem]{Lemma}
\newtheorem{proposition}[theorem]{Proposition}
\newtheorem{remark}[theorem]{Remark}

\newenvironment{proof}[1][Proof]{\noindent\textbf{#1.} }{\ \rule{0.5em}{0.5em}}

\begin{document}

\title{Lyapunov-type inequality for fractional BVP\MakeLowercase{s} involving two Hadamard
fractional derivatives of different orders}
\author{Zaid Laadjal$^{1,*}$ \\
\\
$^{1}$Department of Computer Science, University Center of Illizi, \\
33000 Illizi, Algeria \\
\\
$^{*}$Email: z.laadjal@cuillizi.dz\\
\\
\\
\\
To appear in \textit{\textbf{Journal of Mathematical Inequalities}} \\
 \\ }
\date{}
\maketitle

\begin{abstract}
This paper establishes a Lyapunov-type inequality for a class of fractional
boundary value problems (BVPs) involving two Hadamard fractional derivatives
of different orders with Dirichlet boundary conditions. The method is based
on the construction of the corresponding Green's function and establishing
its maximum value through rigorous analytical techniques. The obtained
inequality provides the necessary conditions for the existence of nontrivial
solutions to the proposed problem. Finally, we illustrate the applicability
of our results by establishing nonexistence criteria for nontrivial
solutions to certain problems and providing examples.
\end{abstract}



\section{Introduction}

Fractional calculus, which extends the concepts of differentiation and
integration to non-integer orders, has become an essential tool in
mathematical modeling across various scientific and engineering disciplines.
Unlike classical integer-order derivatives, fractional derivatives possess
non-local properties and memory effects, making them particularly suitable
for describing complex phenomena in physics, biology, control theory, and
viscoelasticity, etc. (see, e.g., \cite{Cai,Xiao,Pandey,Yang,Pinto}). The
flexibility offered by fractional calculus has led to the development of
numerous fractional derivative operators, each with distinct characteristics
and applications.

Among the various fractional derivative operators, the Hadamard fractional
derivative stands out due to its logarithmic kernel, which makes it
especially useful in problems involving logarithmic functions \cite
{BashirBook}. While much of the existing literature focuses on fractional
BVPs involving a single fractional derivative operator, there is growing
interest in studying problems that incorporate multiple fractional operators
of different orders. Such formulations can better capture the multi-scale
nature of certain physical phenomena.

The qualitative analysis of differential equations has long relied on a
collection of fundamental analytical tools, among which Lyapunov-type
inequalities occupy a prominent place. The foundation of this inequality was
laid in 1907 when the Russian mathematician A. M. Lyapunov \cite{Lyapunov}
studied the classical second-order linear differential equation 
\begin{equation*}
\begin{cases}
u^{\prime \prime }(t)+q(t)x(t)=0, & t_{1}<t< t_{2}, \\ 
x(t_{1})=x(t_{2})=0, & 
\end{cases}
\end{equation*}
where $q:[t_{1},t_{2}]\to\mathbb{R}$ is a continuous function. Lyapunov's
work showed that any nontrivial solution to this problem must satisfy the
following inequality 
\begin{equation*}
\int_{t_{1}}^{t_{2}}|q(\varsigma)|\,d\varsigma >\frac{4}{t_{2}-t_{1}}.
\end{equation*}
This fine and profound result ignited extensive investigations into similar
inequalities for various classes of ordinary differential equations,
culminating in the comprehensive expositions presented in the surveys of
Tiryaki \cite{Tiryaki} and Pinasco \cite{Pinasco}.

In recent years, the extension of Lyapunov-type inequalities to fractional
differential equations has attracted considerable attention. Researchers
have established such inequalities for BVPs involving Riemann-Liouville,
Caputo, Hilfer, and $\psi$-Hilfer derivatives, among others, under diverse
boundary conditions. For a comprehensive treatment of these developments and
their numerous applications, we direct interested readers to the surveys 
\cite{Ntouyas1,Ntouyas2,Ntouyas3,Ntouyas4,Ntouyas5},and the references
therein. Below, we present three prominent studies that constitute the
direct impetus for this work:

Ferreira \cite{Ferreira3} established a Lyapunov-type inequality for the
Riemann-Liouville fractional BVP 
\begin{equation*}
\begin{cases}
(^{\mathcal{R}}\mathfrak{D}_{t_{1}{+}}^{\sigma }x)(t)+q(t)x(t)=0, & 
t_{1}<t<t_{2}, 1<\sigma\leq 2, \\ 
x(t_{1})=x(t_{2})=0. & 
\end{cases}
\end{equation*}
He proved that the existence of a nontrivial continuous solution $u$
necessarily requires the function $q$ to satisfy 
\begin{equation*}
\int_{t_{1}}^{t_{2}}|q(\varsigma)|\,d\varsigma >\frac{4^{\sigma -1}\Gamma
(\sigma)}{ (t_{2}-t_{1})^{\sigma -1}}.
\end{equation*}

A significant advancement in the context of Hadamard fractional derivatives
was achieved by Laadjal et al.~\cite{Zaid2019}, who analyzed the BVP 
\begin{equation*}
\begin{cases}
( ^{\mathcal{H}}\mathfrak{D}_{t_{1}{+}}^{\sigma}x)(t)+q(t)x(t)=0, & 1\leq
t_{1}<t<t_{2}, 1<\sigma\leq 2, \\ 
x(t_{1})=x(t_{2})=0. & 
\end{cases}
\end{equation*}
Their investigation demonstrated that the existence of a nontrivial solution
imposes the integral bound 
\begin{equation*}
\int_{t_{1}}^{t_{2}}|q(\varsigma)|\,d\varsigma\geq \Gamma
(\sigma)\,\varrho\left( \frac{\ln \frac{\varrho}{t_{1}}\times\ln\frac{t_{2}}{
\varrho}}{\ln\frac{t_{2}}{t_{1}}} \right)^{1-\sigma},  \label{50}
\end{equation*}
where $\varrho$ is given by 
\begin{equation*}
\varrho=\exp \left( \frac{1}{2}\left[ 2(\sigma -1)+\ln (t_{1}t_{2})-\sqrt{
4(\sigma -1)^{2}+\ln^{2}\frac{t_{2}}{t_{1}}}\right] \right).
\end{equation*}

Very recently, Silva \cite{Silva2026} investigated a new class of fractional
BVPs involving two Riemann-Liouville fractional derivatives of
different orders:

\begin{equation*}
\begin{cases}
({}^{\mathcal{R}}\mathfrak{D}_{t_{1}{+}}^{\sigma }x)(t)+(^{\mathcal{R}} 
\mathfrak{D}_{t_{1}{+}}^{\kappa }qx)(t)=0, & t_{1}<t< t_{2}, \\ 
x(t_{1})=x(t_{2})=0, & 
\end{cases}
\end{equation*}
where $\, 1<\sigma\leq 2$ and $0<\kappa \leq \sigma -1$. The author
established that any nontrivial solution to this problem must satisfy the
inequality 
\begin{equation*}
\int_{t_{1}}^{t_{2}}|q(\varsigma)|\,d\varsigma > \frac{\Gamma(\sigma-\kappa)
\,(t_{2}-t_{1})^{1-\sigma+\kappa}} { \left( \frac{\sigma-\kappa -1 }{
2\sigma-\kappa -2} \right)^{\sigma-\kappa-1} \left( \frac{\sigma-1}{
2\sigma-\kappa -2}\right)^{\sigma-1}}.
\end{equation*}

These fractional versions not only generalize the classical results but also
prove useful in establishing lower bounds for eigenvalues of fractional
differential operators and in deriving criteria for the absence of
nontrivial solutions. \newline

Inspired by the above-mentioned works, we establish Lyapunov-type inequality
for the following fractional BVPs involving two Hadamard fractional
derivative operators of different orders: 
\begin{equation}
\begin{cases}
({}^{\mathcal{H}}\mathfrak{D}_{t_{1}{+}}^{\sigma }x)(t)+(^{\mathcal{H}} 
\mathfrak{D}_{t_{1}{+}}^{\kappa }qx)(t)=0, & 0<t_{1}<t< t_{2}, \\ 
x(t_{1})=x(t_{2})=0, & 
\end{cases}
\label{PB1}
\end{equation}
where $1<\sigma \leq 2$ and $0<\kappa \leq \sigma -1$; \, ${}^{\mathcal{H}} 
\mathfrak{D}_{t_{1}{+}}^{y}$ denotes the Hadamard fractional derivative of
order $y\in \{\sigma ,\kappa \}$, and $q:[t_{1},t_{2}]\rightarrow \mathbb{R}$
is a continuous function.

The standard methodology for deriving Lyapunov inequalities, applicable to
classical and fractional differential equations, relies on a
well-established analytical technique. The first step consists of reformulating the BVP of
differential equation as an equivalent Fredholm integral equation
characterized by an appropriate kernel function (known as the Green's
function). The second step involves obtaining precise estimates for the
maximum of this kernel across its domain of definition.

To the best of our knowledge, such problems with two Hadamard derivatives
have not been previously investigated in the context of Lyapunov-type
inequalities. Our results were obtained through the explicit construction of Green's
function associated with the given problem and a detailed analysis of its
properties, particularly the determination of its maximum value. \newline

This paper is organized as follows: Section 2 presents the necessary
preliminaries on Hadamard fractional calculus and establishes key properties
that will be used throughout the paper. In Section 3, we construct the
Green's function for the BVP under consideration and derive its maximum
value. Section 4 contains the main results of the paper, where we establish
the Lyapunov-type inequality for the problem under study. Finally, we show
our results' applicability through nonexistence criteria for nontrivial
solutions, with illustrative examples.

\section{Preliminaries}

In this section, we present some basic definitions and properties of
Hadamard fractional calculus that will be essential throughout this paper.
For more details, we refer the reader to \cite{BashirBook,kilbas}.

\begin{definition}
\label{D1} Let $\sigma \geq 0$ and $0<t_{1}< t_{2}<\infty $. The left-sided 
Hadamard fractional integral of order $\sigma $ of a function $x\in 
L^{p}[t_{1},t_{2}], (1\leq p \leq \infty)$, is defined by  
\begin{equation*}
({}^{\mathcal{H}}\mathfrak{I}_{t_{1}{+}}^{0 }x)(t)=x(t),
\end{equation*}
and  
\begin{equation*}
({}^{\mathcal{H}}\mathfrak{I}_{t_{1}{+}}^{\sigma }x)(t)=\frac{1}{\Gamma
(\sigma )}\int_{t_{1}}^{t}\left( \ln \frac{t}{\varsigma}\right) ^{\sigma -1} 
\frac{ x(\varsigma)}{\varsigma}d\varsigma,\quad \text{for } \sigma > 0,
\end{equation*}
where $\Gamma (\cdot )$ denotes the Euler Gamma function.
\end{definition}

\begin{definition}
\label{D2} Let $n\in \mathbb{N}$ and $\sigma \geq 0$ such that $n-1<\sigma 
\leq n $, and let $\delta =t\frac{d}{dt}$ and $AC^n_\delta[t_{1}, t_{2}] = 
\left\{x : [t_{1},t_{2}] \to \mathbb{R} : \delta^{n-1}x \in 
AC[t_{1},t_{2}]\right\}$. The left-sided Hadamard fractional derivative of 
order $\sigma $ of a function $x\in AC^n_\delta[t_{1}, t_{2}]$ is defined by
\begin{equation*}
({}^{\mathcal{H}}\mathfrak{D}_{t_{1}{+}}^{0 }x)(t)=x(t),
\end{equation*}
and  
\begin{equation*}
({}^{\mathcal{H}}\mathfrak{D}_{t_{1}{+}}^{\sigma }x)(t)=\delta^{n}({}^{ 
\mathcal{H}}\mathfrak{I}_{t_{1}{+}}^{n-\sigma }x)(t) \quad \text{for }
\sigma > 0.
\end{equation*}
\end{definition}

\begin{remark}
When $\sigma =n\in \mathbb{N}$, the Hadamard fractional derivative reduces 
to  
\begin{equation*}
({}^{\mathcal{H}}\mathfrak{D}_{t_{1}{+}}^{n}x)(t)=\delta ^{n}x(t).
\end{equation*}
\end{remark}

The following lemmas present fundamental properties of Hadamard fractional
operators.

\begin{lemma}
\label{L1} Let $\sigma ,\kappa >0$, where $n-1<\sigma \leq n$ with $n\in  
\mathbb{N}$. Then, the following properties hold:

\begin{enumerate}
\item $({}^{\mathcal{H}}\mathfrak{I}_{t_{1}{+}}^{\sigma }{}^{\mathcal{H}}  
\mathfrak{I}_{t_{1}{+}}^{\kappa }x)(t)=({}^{\mathcal{H}}\mathfrak{I} _{t_{1}{
\ \ +}}^{\sigma +\kappa }x)(t)$,

\item $({}^{\mathcal{H}}\mathfrak{D}_{t_{1}{+}}^{\sigma }{}^{\mathcal{H}}  
\mathfrak{I}_{t_{1}{+}}^{\sigma }x)(t)=x(t)$,

\item $({}^{\mathcal{H}}\mathfrak{I}_{t_{1}{+}}^{\sigma }{}^{\mathcal{H}}  
\mathfrak{D}_{t_{1}{+}}^{\sigma }x)(t)=x(t)-\sum_{j=1}^{n}c_{j}\left( \ln  
\frac{t}{t_{1}}\right) ^{\sigma -j},$ \quad $c_{j}\in \mathbb{R}$ $
j=1,2,\ldots ,n$. 
\end{enumerate}
\end{lemma}

\begin{lemma}
\label{L2} For $\sigma ,\kappa >0$, we have  
\begin{equation*}
{}^{\mathcal{H}}\mathfrak{I}_{t_{1}{+}}^{\sigma }\left( \ln \frac{t}{t_{1}}
\right) ^{\kappa -1}=\frac{\Gamma (\kappa )}{\Gamma (\kappa +\sigma )}
\left( \ln \frac{t}{t_{1}}\right) ^{\kappa +\sigma -1},
\end{equation*}

\begin{equation*}
{}^{\mathcal{H}}\mathfrak{D}_{t_{1}{+}}^{\sigma }\left( \ln \frac{t}{t_{1}}
\right) ^{\kappa -1}=\frac{\Gamma (\kappa )}{\Gamma (\kappa -\sigma )}
\left( \ln \frac{t}{t_{1}}\right) ^{\kappa -\sigma -1}.
\end{equation*}
\end{lemma}

\section{Green's Function Construction and Maximum Analysis}

In this section, we construct the Green's function associated with the
fractional BVP (\ref{PB1}) and establish its fundamental properties. We
do this by first transforming the BVP into an integral operator
formulation involving an a Green's function, which will serve as a crucial
tool for obtaining our main results.

\begin{lemma}
\label{LINT} Let $1<\sigma \leq 2$. Then, $x\in C([t_{1},t_{2}],\mathbb{R})$ is a 
solution of the fractional boundary value problem (\ref{PB1}) if, and only 
if, $x$ satisfies the integral equation  
\begin{equation}
x(t)=\int_{t_{1}}^{t_{2}}G(t,\varsigma )\,q(\varsigma )\,x(\varsigma
)\,d\varsigma ,\quad t\in \lbrack t_{1},t_{2}],  \label{eqv}
\end{equation}
where the Green's function $G:[t_{1},t_{2}]\times \lbrack 
t_{1},t_{2}]\rightarrow \mathbb{R}$ has the form  
\begin{equation}
G(t,\varsigma )=\frac{1}{\Gamma (\sigma -\kappa )} 
\begin{cases}
\Xi _{1}(t,\varsigma ), & t_{1}\leq t\leq \varsigma \leq t_{2}, \\ 
\Xi _{2}(t,\varsigma ), & t_{1}\leq \varsigma \leq t\leq t_{2},
\end{cases}
\label{gr}
\end{equation}
with  
\begin{align}
\Xi _{1}(t,\varsigma )& =\frac{1}{\varsigma }\times \frac{\left( \ln \frac{t 
}{t_{1}}\right) ^{\sigma -1}\left( \ln \frac{t_{2}}{\varsigma }\right)
^{\sigma -\kappa -1}}{\left( \ln \frac{t_{2}}{t_{1}}\right) ^{\sigma -1}},
\label{g1} \\
\Xi _{2}(t,\varsigma )& =\frac{1}{\varsigma }\times \frac{\left( \ln \frac{t 
}{t_{1}}\right) ^{\sigma -1}\left( \ln \frac{t_{2}}{\varsigma }\right)
^{\sigma -\kappa -1}}{\left( \ln \frac{t_{2}}{t_{1}}\right) ^{\sigma -1}}- 
\frac{1}{\varsigma }\left( \ln \frac{t}{\varsigma }\right) ^{\sigma -\kappa
-1}.  \label{g2}
\end{align}
\end{lemma}

\begin{proof}
According to the first property of Lemma \ref{L1}, we have  
\begin{equation}
{}^{\mathcal{H}}\mathfrak{I}_{t_{1}{+}}^{\sigma }\circ \,{}^{\mathcal{H}} 
\mathfrak{D}_{t_{1}{+}}^{\kappa }=\,{}^{\mathcal{H}}\mathfrak{I}_{t_{1}{+}
}^{\sigma -\kappa }\circ \,{}^{\mathcal{H}}\mathfrak{I}_{t_{1}{+}}^{\kappa
}\circ \,{}^{\mathcal{H}}\mathfrak{D}_{t_{1}{+}}^{\kappa }.  \label{oper}
\end{equation}

Applying this composition rule to the function $qx$ and utilizing the third 
property of Lemma \ref{L1} and then using Lemma \ref{L2}, we obtain  
\begin{equation}
{}^{\mathcal{H}}\mathfrak{I}_{t_{1}{+}}^{\sigma }\left( {}^{\mathcal{H}} 
\mathfrak{D}_{t_{1}{+}}^{\kappa }(qx)\right) (t)= \,{}^{\mathcal{H}}\mathfrak{I}_{t_{1}{+}}^{\sigma -\kappa }(qx)(t)+\lambda \left( \ln \frac{t}{t_{1}}
\right) ^{\sigma -1},  \label{eq06}
\end{equation}
where $\lambda =c_{1}\Gamma (\kappa )/\Gamma (\sigma )$ for some constant $
c_{1}\in \mathbb{R}$.

Next, we apply the operator ${}^{\mathcal{H}}\mathfrak{I}_{t_{1}{+}}^{\sigma
}$ to both sides of the differential equation  
\begin{equation}
{}^{\mathcal{H}}\mathfrak{D}_{t_{1}{+}}^{\sigma }x(t)=-{}^{\mathcal{H}} 
\mathfrak{D}_{t_{1}{+}}^{\kappa }(qx)(t),  \label{eq07}
\end{equation}
we obtain  
\begin{equation}
x(t)=\rho _{1}\left( \ln \frac{t}{t_{1}}\right) ^{\sigma -1}+\rho _{2}\left(
\ln \frac{t}{t_{1}}\right) ^{\sigma -2}-\,{}^{\mathcal{H}}\mathfrak{I}
_{t_{1} {+}}^{\sigma -\kappa }(qx)(t)-\lambda \,\left( \ln \frac{t}{t_{1}}
\right) ^{\sigma -1},  \label{gen1}
\end{equation}
where $\rho _{1},\rho _{2}\in \mathbb{R}$. By taking $\xi _{1}=\rho 
_{1}-\lambda $ and $\xi _{2}=\rho _{2}$, equation (\ref{gen1}) becomes  
\begin{equation}
x(t)=\xi _{1}\left( \ln \frac{t}{t_{1}}\right) ^{\sigma -1}+\xi _{2}\left(
\ln \frac{t}{t_{1}}\right) ^{\sigma -2}-\,{}^{\mathcal{H}}\mathfrak{I}
_{t_{1} {+}}^{\sigma -\kappa }(qx)(t).  \label{gen2}
\end{equation}

Using the boundary conditions $x(t_{1})=0$, and $x(t_{2})=0$, we get $\xi 
_{2}=0$ and  
\begin{equation*}
\xi _{1}=\frac{1}{\left( \ln \frac{t_{2}}{t_{1}}\right) ^{\sigma -1}\Gamma
(\sigma -\kappa )}\int_{t_{1}}^{t_{2}}\left( \ln \frac{t_{2}}{\varsigma }
\right) ^{\sigma -\kappa -1}\frac{q(\varsigma )x(\varsigma )}{\varsigma }
\,d\varsigma .
\end{equation*}

Substituting the values of $\xi_{1}$ and $\xi_{2}$ in (\ref{gen2}), we 
obtain

\begin{align}
x(t)& =\frac{\left( \ln \frac{t}{t_{1}}\right) ^{\sigma -1}}{\left( \ln 
\frac{t_{2}}{t_{1}}\right) ^{\sigma -1}\Gamma (\sigma -\kappa )}
\int_{t_{1}}^{t_{2}}\left( \ln \frac{t_{2}}{\varsigma}\right) ^{\sigma
-\kappa -1} \frac{q(\varsigma )x(\varsigma )}{\varsigma }\,d\varsigma  \notag
\\
& \quad -\frac{1}{\Gamma (\sigma -\kappa )}\int_{t_{1}}^{t}\left( \ln \frac{
t }{\varsigma }\right) ^{\sigma -\kappa -1}\frac{q(\varsigma )x(\varsigma )}{
\varsigma }\,d\varsigma .  \label{deco}
\end{align}
We decompose the first integral in (\ref{deco}) as follows  
\begin{eqnarray*}
\int_{t_{1}}^{t_{2}}\left( \ln \frac{t_{2}}{\varsigma}\right) ^{\sigma
-\kappa -1} \frac{q(\varsigma )x(\varsigma )}{\varsigma }\,d\varsigma
&=&\int_{t_{1}}^{t}\left( \ln \frac{t_{2}}{\varsigma}\right) ^{\sigma
-\kappa -1} \frac{q(\varsigma )x(\varsigma )}{\varsigma }\,d\varsigma \\
&&+\int_{t}^{t_{2}}\left( \ln \frac{t_{2}}{\varsigma}\right) ^{\sigma
-\kappa -1} \frac{q(\varsigma )x(\varsigma )}{\varsigma }\,d\varsigma .
\end{eqnarray*}
Substituting this decomposition into equation (\ref{deco}), we  obtain the
desired integral equation (\ref{eqv}).

Conversely, it is clear that if $x$ satisfies equation (\ref{eqv}), then  (
\ref{PB1}) follows directly. The proof is ended.
\end{proof}

\bigskip

The subsequent theorem provides sharp bounds on the Green's function, which
serves as the cornerstone for deriving our Lyapunov-type inequality.

\begin{proposition}
\label{max0} For the Green's function $G$ defined by (\ref{gr})--(\ref{g2}),
we have 
\begin{equation}
\max_{(t,\varsigma )\in \lbrack t_{1},t_{2}]^{2}}|G(t,\varsigma )|=\frac{
\max \left\{ \Omega ,\,\mho \right\} }{\Gamma (\sigma -\kappa )},
\label{max_estimate}
\end{equation}
where 
\begin{align}
\Omega & =\frac{X_{2}^{\sigma -1}\left( \ln \frac{t_{2}}{t_{1}}
-X_{2}\right) ^{\sigma -\kappa -1}}{\left( \ln \frac{t_{2}}{t_{1}}\right)
^{\sigma -1}\,t_{1}\,e^{X_{2}}},  \label{omeg} \\
\mho & =\frac{\kappa }{\sigma -1}\left( 1-\frac{\kappa }{\sigma -1}\right)
^{\frac{\sigma -\kappa -1}{\kappa }}\frac{1}{t_{1}}\left( \ln \frac{t_{2}}{
t_{1}}\right) ^{\sigma -\kappa -1},  \label{mho}
\end{align}
with $X_{2}\in (0,\ln \frac{t_{2}}{t_{1}})$, given by 
\begin{equation}
X_{2}=\frac{1}{2}\left( \ln \frac{t_{2}}{t_{1}}+2(\sigma -1)-\kappa -\sqrt{
\Delta }\,\right) ,\quad \Delta =\left( \ln \frac{t_{2}}{t_{1}}-\kappa
\right) ^{2}+4(\sigma -1)\left( \sigma -1-\kappa \right) .  \label{x21}
\end{equation}
\end{proposition}

\begin{proof}
We split the analysis into two cases based on the relative position of $t$ 
and $\varsigma$.

\textbf{Case 1: $t_{1}\leq t\leq \varsigma \leq t_{2}$.}

In this case, the Green's function takes the form  
\begin{equation}
G(t,\varsigma)=\frac{\Xi_{1}(t,\varsigma)}{\Gamma (\sigma -\kappa )},
\label{case1_green}
\end{equation}
where $\Xi_{1}(t,\varsigma)$ is given by (\ref{g1}). Clearly, $
\Xi_{1}(t,\varsigma)\geq 0$ for all $(t,\varsigma)$ in this domain, with $
\Xi_{1}(t,t_2)=0$.  To locate the maximum, we fix $t\in \lbrack t_{1},t_{2}]$
, and for $\varsigma\in [t,t_2)$ compute the  partial derivative  
\begin{align}
\frac{\partial \Xi _{1}(t,\varsigma )}{\partial \varsigma }& =-\frac{\left(
\ln \frac{t}{t_{1}}\right) ^{\sigma -1}}{\left( \ln \frac{t_{2}}{t_{1}}
\right) ^{\sigma -1}}\times \frac{1}{\varsigma ^{2}}\left[ \left( \ln \frac{
t_{2}}{\varsigma }\right) ^{\sigma -\kappa -1}+(\sigma -\kappa -1)\left( \ln 
\frac{t_{2}}{\varsigma }\right) ^{\sigma -\kappa -2}\right]  \notag \\
& \leq 0.  \label{dg11}
\end{align}
This shows that $\Xi _{1}(t,\varsigma )$ is decreasing with respect to $
\varsigma $ for each fixed $t$. Therefore,  
\begin{equation}
0\leq \Xi _{1}(t,\varsigma )\leq \Xi _{1}(t,t)=\Xi _{2}(t,t),\quad
\,t_{1}\leq t\leq \varsigma \leq t_{2}.  \label{comp1}
\end{equation}

Define the auxiliary function  
\begin{equation}
h(t)=\frac{\left( \ln \frac{t}{t_{1}}\right) ^{\sigma -1}\left( \ln \frac{
t_{2}}{t}\right) ^{\sigma -\kappa -1}}{t},\quad t\in \lbrack t_{1},t_{2}],
\label{h_def}
\end{equation}
so  
\begin{equation*}
\Xi _{1}(t,t)=\frac{h(t)}{\left( \ln \frac{t_{2}}{t_{1}}\right) ^{\sigma -1}}
.
\end{equation*}
Observe that $h(t_{1})=h(t_{2})=0$ and $h(t)$ is a continuous function with $
h(t)\geq 0$ for all $t\in \lbrack t_{1},t_{2}]$. Then, there exists a point $
t^{\ast }\in (t_{1},t_{2})$ where $h$ attains its maximum. So, to find the 
maximum value, computing the derivative of $h$ on $(t_{1},t_{2})$, we get  
\begin{align}
h^{\prime }(t)& =-\frac{1}{t^{2}}\left( \ln \frac{t}{t_{1}}\right) ^{\sigma
-1}\left( \ln \frac{t_{2}}{t}\right) ^{\sigma -\kappa -1}+\frac{1}{t^{2}} 
\left[ (\sigma -1)\left( \ln \frac{t}{t_{1}}\right) ^{\sigma -2}\left( \ln 
\frac{t_{2}}{t}\right) ^{\sigma -\kappa -1}\right.  \notag \\
& \quad -\left. (\sigma -\kappa -1)\left( \ln \frac{t}{t_{1}}\right)
^{\sigma -1}\left( \ln \frac{t_{2}}{t}\right) ^{\sigma -\kappa -2}\right] ,
\label{h_prime}
\end{align}
we get  
\begin{align}
h^{\prime }(t)& =\frac{1}{t^{2}}\left( \ln \frac{t}{t_{1}}\right) ^{\sigma
-2}\left( \ln \frac{t_{2}}{t}\right) ^{\sigma -\kappa -2}  \notag \\
& \quad \times \left[ (\sigma -1)\left( \ln \frac{t_{2}}{t}\right) -\left(
\ln \frac{t}{t_{1}}\right) \left( \ln \frac{t_{2}}{t}\right) -(\sigma
-\kappa -1)\left( \ln \frac{t}{t_{1}}\right) \right] .  \label{fact9}
\end{align}

Taking $h^{\prime }(t)=0$ and using the equality $\ln \frac{t_{2}}{t}=\ln 
\frac{t_{2}}{t_{1}}-\ln \frac{t}{t_{1}}$,\ and introducing the change of
variable $X=\ln \frac{t}{t_{1}}$, (here $X\in (0,\ln \frac{t_{2}}{t_{1}})$
), the critical points equation becomes 
\begin{equation}
-X\left( \ln \frac{t_{2}}{t_{1}}-X\right) +(\sigma -1)\left( \ln \frac{t_{2}
}{t_{1}}-X\right) -(\sigma -\kappa -1)X=0.  \label{crit_eq1}
\end{equation}
Expanding and simplifying yields 
\begin{equation}
X^{2}-\left( \ln \frac{t_{2}}{t_{1}}+2(\sigma -1)-\kappa \right) X+(\sigma
-1)\ln \frac{t_{2}}{t_{1}}=0.  \label{2x}
\end{equation}
The discriminant of this quadratic equation is 
\begin{align}
\Delta & =\left( \ln \frac{t_{2}}{t_{1}}+2(\sigma -1)-\kappa \right)
^{2}-4(\sigma -1)\ln \frac{t_{2}}{t_{1}}  \notag \\
& =\left( \ln \frac{t_{2}}{t_{1}}\right) ^{2}+4(\sigma -1)^{2}+\kappa
^{2}+4(\sigma -1)\ln \frac{t_{2}}{t_{1}}-4(\sigma -1)\kappa -2\kappa \ln 
\frac{t_{2}}{t_{1}}  \notag \\
& \quad -4(\sigma -1)\ln \frac{t_{2}}{t_{1}}  \notag \\
& =(\ln \frac{t_{2}}{t_{1}}-\kappa )^{2}+4(\sigma -1)\left( \sigma -1-\kappa
\right) >0.  \label{discrim}
\end{align}

Thus, equation (\ref{2x}) has two distinct real solutions:  
\begin{align}
X_{1}& =\frac{1}{2}\left( \ln \frac{t_{2}}{t_{1}}+2(\sigma -1)-\kappa +\sqrt{
\Delta }\right) ,  \label{x1} \\
X_{2}& =\frac{1}{2}\left( \ln \frac{t_{2}}{t_{1}}+2(\sigma -1)-\kappa -\sqrt{
\Delta }\right) .  \label{x2}
\end{align}

However, by using the value of $\Delta$ found in the first line of (\ref{discrim}), we get
\begin{equation*}
X_{1}>\frac{1}{2}\left( \ln \frac{t_{2}}{t_{1}}+\sqrt{\left( \ln \frac{t_{2} 
}{t_{1}}\right) ^{2}}\right) =\ln \frac{t_{2}}{t_{1}},
\end{equation*}
which implies $X_{1}\notin (0,\ln \frac{t_{2}}{t_{1}})$. Hence, $X_{2}$ is 
the the unique solution of the equation (\ref{2x}) on $(0,\ln \frac{ t_{2}}{
t_{1}})$.

Consequently, the unique admissible critical point of the the function $h$ 
is $t^{\ast }=t_{1}e^{X_{2}}\in (t_{1},t_{2})$, and we obtain  
\begin{equation}
\underset{t_{1}\leq t\leq t_{2}}{\max }\Xi _{1}(t,t)=\frac{h(t^{\ast })}{
\left( \ln \frac{t_{2}}{t_{1}}\right) ^{\sigma -1}}=\frac{X_{2}^{\sigma
-1}\left( \ln \frac{t_{2}}{t_{1}}-X_{2}\right) ^{\sigma -\kappa -1}}{\left(
\ln \frac{t_{2}}{t_{1}}\right) ^{\sigma -1}\,t_{1}\,e^{X_{2}}},  \label{g31}
\end{equation}
where $X_{2}$ is given by (\ref{x2}). \newline

\textbf{Case 2: $t_{1}\leq \varsigma \leq t\leq t_{2}$.}

In this case ($t_{1}\leq \varsigma \leq t\leq t_{2}$), the Green's function 
is given by  
\begin{equation}
G(t,\varsigma)=\frac{\Xi_{2}(t,\varsigma)}{\Gamma (\sigma -\kappa )}.
\label{case2_green}
\end{equation}

Note that, if $t=\varsigma$ we have $G(t,t)=\frac{\Xi_{2}(t,t)}{\Gamma
(\sigma -\kappa )}=\frac{\Xi_{1}(t,t)}{\Gamma (\sigma -\kappa )}.$

For $t<\varsigma ,$ we examine the monotonicity of $\Xi _{2}$ with respect 
to $\varsigma $ for fixed $t\in \lbrack t_{1},t_{2}]$. We have  
\begin{align}
\frac{\partial \Xi _{2}(t,\varsigma )}{\partial \varsigma }& =-\frac{\left(
\ln \frac{t}{t_{1}}\right) ^{\sigma -1}}{\varsigma ^{2}\left( \ln \frac{
t_{2} }{t_{1}}\right) ^{\sigma -1}}\left( \ln \frac{t_{2}}{\varsigma }
\right) ^{\sigma -\kappa -1}-\frac{(\sigma -\kappa -1)\left( \ln \frac{t}{
t_{1}} \right) ^{\sigma -1}}{\varsigma ^{2}\left( \ln \frac{t_{2}}{t_{1}}
\right) ^{\sigma -1}}\left( \ln \frac{t_{2}}{\varsigma }\right) ^{\sigma
-\kappa -2}  \notag \\
& \quad +\frac{1}{\varsigma ^{2}}\left( \ln \frac{t}{\varsigma }\right)
^{\sigma -\kappa -1}+\frac{(\sigma -\kappa -1)}{\varsigma ^{2}}\left( \ln 
\frac{t}{\varsigma }\right) ^{\sigma -\kappa -2},  \label{dg22}
\end{align}
we get  
\begin{align}
\frac{\partial \Xi _{2}(t,\varsigma )}{\partial \varsigma }& =\frac{\left(
\ln \frac{t_{2}}{\varsigma }\right) ^{\sigma -\kappa -2}}{\varsigma ^{2}} 
\Bigg[\left( \ln \frac{t_{2}}{\varsigma }\right) \left( \frac{\left( \ln 
\frac{t}{\varsigma }\right) ^{\sigma -\kappa -1}}{\left( \ln \frac{t_{2}}{
\varsigma }\right) ^{\sigma -\kappa -1}}-\frac{\left( \ln \frac{t}{t_{1}}
\right) ^{\sigma -1}}{\left( \ln \frac{t_{2}}{t_{1}}\right) ^{\sigma -1}}
\right)  \notag \\
& \quad +(\sigma -\kappa -1)\left( \frac{\left( \ln \frac{t}{\varsigma }
\right) ^{\sigma -\kappa -2}}{\left( \ln \frac{t_{2}}{\varsigma }\right)
^{\sigma -\kappa -2}}-\frac{\left( \ln \frac{t}{t_{1}}\right) ^{\sigma -1}}{
\left( \ln \frac{t_{2}}{t_{1}}\right) ^{\sigma -1}}\right) \Bigg].
\label{dg2}
\end{align}

To determine the sign of this derivative, we establish key inequalities. 
Note that  
\begin{align}
\frac{\ln \frac{t}{\varsigma}}{\ln \frac{t_{2}}{\varsigma}}&=\frac{\ln \frac{
t_{2}}{\varsigma}-\ln \frac{t_{2}}{t}}{\ln \frac{t_{2}}{\varsigma}}=1-\frac{
\ln \frac{t_{2}}{t}}{\ln \frac{ t_{2}}{\varsigma}}  \notag \\
&>1-\frac{\ln \frac{t_{2}}{t}}{\ln \frac{t_{2}}{t_{1}}}=\frac{\ln \frac{
t_{2} }{t_{1}}-\ln \frac{t_{2}}{t}}{\ln \frac{t_{2}}{t_{1}}}=\frac{\ln \frac{
t}{ t_{1}}}{\ln \frac{t_{2}}{t_{1}}}.  \label{ineq801}
\end{align}

Since $0 \leq \sigma -\kappa -1<\sigma -1 \leq 1$ and the ratio satisfies $
0< \frac{\ln \frac{t}{t_{1}}}{\ln \frac{t_{2}}{t_{1}}}<\frac{\ln \frac{t}{
\varsigma}}{\ln \frac{ t_{2}}{\varsigma}}<1$, we deduce  
\begin{equation}
\left( \frac{\ln \frac{t}{\varsigma}}{\ln \frac{t_{2}}{\varsigma}}\right)
^{\sigma -\kappa -1}>\left( \frac{\ln \frac{t}{\varsigma}}{\ln \frac{t_{2}}{
\varsigma}}\right) ^{\sigma -1}>\left( \frac{\ln \frac{t}{t_{1}}}{\ln \frac{
t_{2}}{t_{1}}}\right) ^{\sigma -1}.  \label{ineq1}
\end{equation}

Furthermore, since $\sigma -\kappa -2<0$ and $\sigma -1>0$, we conclude that
\begin{equation}
\frac{\left( \ln \frac{t}{\varsigma }\right) ^{\sigma -\kappa -2}}{\left(
\ln \frac{t_{2}}{\varsigma }\right) ^{\sigma -\kappa -2}}>1\quad \text{and}
\quad \frac{\left( \ln \frac{t}{t_{1}}\right) ^{\sigma -1}}{\left( \ln \frac{
t_{2}}{t_{1}}\right) ^{\sigma -1}}<1,  \label{ineq90}
\end{equation}
which implies  
\begin{equation}
\frac{\left( \ln \frac{t}{\varsigma }\right) ^{\sigma -\kappa -2}}{\left(
\ln \frac{t_{2}}{\varsigma }\right) ^{\sigma -\kappa -2}}-\frac{\left( \ln 
\frac{t}{t_{1}}\right) ^{\sigma -1}}{\left( \ln \frac{t_{2}}{t_{1}}\right)
^{\sigma -1}}>0.  \label{ineq0}
\end{equation}

Combining (\ref{ineq1}) and (\ref{ineq0}), both terms in the  brackets of (
\ref{dg2}) are positive, hence  
\begin{equation}
\frac{\partial \Xi_{2}(t,\varsigma)}{\partial \varsigma}\geq 0.
\label{g2incr}
\end{equation}
This establishes that $\Xi_{2}(t,\varsigma)$ is increasing with respect to $
\varsigma$. So, we get  
\begin{equation}
\Xi_{2}(t,t_{1})\leq \Xi_{2}(t,\varsigma)\leq \Xi_{2}(t,t),\quad \text{for
all } \varsigma \in \lbrack t_{1},t].  \label{g2b}
\end{equation}

Consequently, in the case $t_{1}\leq \varsigma \leq t\leq t_{2}$, we have  
\begin{equation}
\max_{t_{1}\leq \varsigma \leq t\leq t_{2}}|\Xi_{2}(t,\varsigma)|=\max
\left\{ \max_{t\in \lbrack t_{1},t_{2}]}|\Xi_{2}(t,t_{1})|,\,\max_{t\in
[t_{1},t_{2}]}|\Xi_{2}(t,t)|\right\} .  \label{max_g2}
\end{equation}

Since we already established that $\Xi _{2}(t,t)=\Xi _{1}(t,t)\geq 0$ in 
Case 1, it remains to analyze $\Xi _{2}(t,t_{1})$. We have  
\begin{align}
\Xi _{2}(t,t_{1})& =\frac{\left( \ln \frac{t}{t_{1}}\right) ^{\sigma
-1}\left( \ln \frac{t_{2}}{t_{1}}\right) ^{\sigma -\kappa -1}}{t_{1}\left(
\ln \frac{t_{2}}{t_{1}}\right) ^{\sigma -1}}-\frac{\left( \ln \frac{t}{t_{1}}
\right) ^{\sigma -\kappa -1}}{t_{1}}  \notag \\
& =\frac{\left( \ln \frac{t}{t_{1}}\right) ^{\sigma -1}}{t_{1}}\left[ \left(
\ln \frac{t_{2}}{t_{1}}\right) ^{-\kappa }-\left( \ln \frac{t}{t_{1}}\right)
^{-\kappa }\right]  \notag \\
& =\frac{\left( \ln \frac{t}{t_{1}}\right) ^{\sigma -\kappa -1}}{t_{1}}\left[
\left( \frac{\ln \frac{t}{t_{1}}}{\ln \frac{t_{2}}{t_{1}}}\right) ^{\kappa
}-1\right] \leq 0.  \label{gtt1}
\end{align}

Therefore,  
\begin{equation}
|\Xi _{2}(t,t_{1})|=\frac{\left( \ln \frac{t}{t_{1}}\right) ^{\sigma -\kappa
-1}}{t_{1}}\left[ 1-\left( \frac{\ln \frac{t}{t_{1}}}{\ln \frac{t_{2}}{t_{1}}
}\right) ^{\kappa }{\ }\right] .  \label{abs2}
\end{equation}

Define the auxiliary function  
\begin{equation}
\zeta (t)=|\Xi _{2}(t,t_{1})|=\frac{\left( \ln \frac{t}{t_{1}}\right)
^{\sigma -\kappa -1}}{t_{1}}\left[ 1-\left( \frac{\ln \frac{t}{t_{1}}}{\ln 
\frac{t_{2}}{t_{1}}}\right) ^{\kappa }{\ }\right] ,\quad t\in \lbrack
t_{1},t_{2}].  \label{zet2}
\end{equation}

Observe that $\zeta (t_{1})=\zeta (t_{2})=0$ and $\zeta (t)$ is a continuous
function with $\zeta (t)\geq 0$ for all $t\in \lbrack t_{1},t_{2}]$. So, to 
find the maximum value, we compute the derivative of $\zeta $ for $t\in 
(t_{1},t_{2})$:  
\begin{align}
\zeta ^{\prime }(t)& =\frac{1}{t_{1}t}\left[ (\sigma -\kappa -1)\left( \ln 
\frac{t}{t_{1}}\right) ^{\sigma -\kappa -2}-(\sigma -1)\frac{\left( \ln 
\frac{t}{t_{1}}\right) ^{\sigma -2}}{\left( \ln \frac{t_{2}}{t_{1}}\right)
^{\kappa }}\right]  \notag \\
& =\frac{\left( \ln \frac{t}{t_{1}}\right) ^{\sigma -2}}{t_{1}t}\left[ \frac{
\sigma -\kappa -1}{\left( \ln \frac{t}{t_{1}}\right) ^{\kappa }}-\frac{
\sigma -1}{\left( \ln \frac{t_{2}}{t_{1}}\right) ^{\kappa }}\right] .
\label{zeta_prime}
\end{align}

Taking $\zeta ^{\prime }(\hat{t})=0$ yields the critical point condition  
\begin{equation}
\frac{\sigma -\kappa -1}{\left( \ln \frac{\hat{t}}{t_{1}}\right) ^{\kappa }}
= \frac{\sigma -1}{\left( \ln \frac{t_{2}}{t_{1}}\right) ^{\kappa }},
\label{critical_cond}
\end{equation}
from which we obtain  
\begin{equation}
\hat{t}=t_{1}\exp \left[ \left( \frac{\sigma -\kappa -1}{\sigma -1}\right)
^{ \frac{1}{\kappa }}\ln \frac{t_{2}}{t_{1}}\right] .  \label{t0}
\end{equation}

Since $0<\frac{\sigma -\kappa -1}{\sigma -1}<1$, we have $0<\left( \frac{
\sigma -\kappa -1}{\sigma -1}\right) ^{\frac{1}{\kappa }}<1$, which ensures $
t_{1}<\hat{t}<t_{2}$.

Consequently, we find  
\begin{align}
\max_{t\in \lbrack t_{1},t_{2}]}|\Xi _{2}(t,t_{1})|& =\max_{t\in \lbrack
t_{1},t_{2}]}|\zeta (t)|=\zeta (\hat{t})  \notag \\
& =\frac{1}{t_{1}}\left[ \left( \frac{\sigma -\kappa -1}{\sigma -1}\right)
^{ \frac{1}{\kappa }}\ln \frac{t_{2}}{t_{1}}\right] ^{\sigma -\kappa -1}
\left[ 1-\frac{\sigma -\kappa -1}{\sigma -1}\right]  \notag \\
& =\frac{1}{t_{1}}\left( \frac{\sigma -\kappa -1}{\sigma -1}\right) ^{\frac{
\sigma -\kappa -1}{\kappa }}\frac{\kappa }{\sigma -1}\left( \ln \frac{t_{2}}{
t_{1}}\right) ^{\sigma -\kappa -1}.  \label{zet1}
\end{align}

Rearranging the expression, we obtain  
\begin{equation}
\max_{t\in \lbrack t_{1},t_{2}]}|\Xi _{2}(t,t_{1})|=\frac{1}{t_{1}\Gamma
(\sigma -\kappa )}\left( 1-\frac{\kappa }{\sigma -1}\right) ^{\frac{\sigma
-\kappa -1}{\kappa }}\frac{\kappa }{\sigma -1}\left( \ln \frac{t_{2}}{t_{1}}
\right) ^{\sigma -\kappa -1}.  \label{g23}
\end{equation}

Combining the maximum values from both cases, specifically equations (\ref
{comp1}), (\ref{g31}) and (\ref{g23}), we conclude that for $\sigma \in 
(1,2)$ and $\kappa \in (0,\sigma -1)$, the maximum of the Green's function 
over $[t_{1},t_{2}]\times \lbrack t_{1},t_{2}]$ is given by  
\begin{equation}
\max_{(t,\varsigma )\in \lbrack t_{1},t_{2}]^{2}}|G(t,\varsigma )|=\frac{
\max \left\{ \Omega ,\,\mho \right\} }{\Gamma (\sigma -\kappa )},
\label{final_max}
\end{equation}
where  
\begin{align*}
\Omega & =\frac{X_{2}^{\sigma -1}\left( \ln \frac{t_{2}}{t_{1}}
-X_{2}\right) ^{\sigma -\kappa -1}}{\left( \ln \frac{t_{2}}{t_{1}}\right)
^{\sigma -1}\,t_{1}\,e^{X_{2}}}, \\
\mho & =\frac{1}{t_{1}}\left( 1-\frac{\kappa }{\sigma -1}\right) ^{\frac{
\sigma -\kappa -1}{\kappa }}\frac{\kappa }{\sigma -1}\left( \ln \frac{t_{2}}{
t_{1}}\right) ^{\sigma -\kappa -1},
\end{align*}
with $X_{2}$ defined as in (\ref{x2}). This completes the proof.
\end{proof}

\section{Lyapunov-type inequality for the Problem at Hand}

We now present the main Lyapunov-type inequality for the Hadamard fractional
BVP (\ref{PB1}). This result provides a necessary condition for the
existence of nontrivial solutions.

\begin{theorem}
\label{t2} Suppose that problem (\ref{PB1}) has a nontrivial continuous 
solution. Then, the following inequality holds:  
\begin{equation}
\int_{t_{1}}^{t_{2}}|q(\varsigma)|\,d\varsigma \geq \frac{\Gamma (\sigma
-\kappa )}{\max \left\{ \Omega ,\, \mho \right\}},  \label{36}
\end{equation}
where $\max \left\{ \Omega ,\, \mho \right\}$ is a constant defined in 
Proposition \ref{max0}.
\end{theorem}

\begin{proof}
Consider the Banach space $E=C([t_{1},t_{2}],\mathbb{R})$ equipped with the 
supremum norm  
\begin{equation}
\|x\|_{\infty} =\sup_{t\in [t_{1},t_{2}]}|x(t)|.  \label{sup1}
\end{equation}

Let $x$ be a nontrivial continuous solution of problem (\ref{PB1}). By Lemma 
\ref{LINT}, $x$ satisfies the integral equation (\ref{eqv}). Taking 
absolute values, we obtain  
\begin{equation}
|x(t)|\leq\int_{t_{1}}^{t_{2}}|G(t,\varsigma)|\,|q(\varsigma)|\,|x(
\varsigma)|\,d\varsigma.  \label{pbound}
\end{equation}
From this we get  
\begin{equation}
|x(t)|\leq \max_{(t,\varsigma)\in[t_{1},t_{2}]^{2}}|G(t,\varsigma)|\,
\int_{t_{1}}^{t_{2}}|q(\varsigma)|\,|x(\varsigma)|\,d\varsigma.  \label{bnd1}
\end{equation}

Taking the supremum over $t\in[t_{1},t_{2}]$ on both sides of (\ref{bnd1}),
we derive  
\begin{align}
\|x\|_{\infty}\leq \|x\|_{\infty}\max_{(t,\varsigma)\in[t_{1},t_{2}]
^{2}}|G(t,\varsigma)| \int_{t_{1}}^{t_{2}}|q(\varsigma)|\,d\varsigma.
\label{nrm2}
\end{align}

Since $x$ is a nontrivial solution, (i.e., $\Vert x\Vert _{\infty }>0$). 
Dividing both sides of (\ref{nrm2}) by $\Vert x\Vert _{\infty }$ yields  
\begin{equation}
1\leq \max_{(t,\varsigma)\in [t_{1},t_{2}]^{2}}|G(t,\varsigma)|
\int_{t_{1}}^{t_{2}}|q(\varsigma)|\,d\varsigma.  \label{key1}
\end{equation}

Using Proposition \ref{max0}, we obtain the desired inequality (\ref{36}).
\end{proof}

\bigskip

Theorem \ref{t2} yields directly the following corollary.

\begin{corollary}
\label{co2} If the following problem  
\begin{equation}
\begin{cases}
({}^{\mathcal{H}}\mathfrak{D}_{t_{1}{+}}^{\sigma }x)(t)+\lambda (^{\mathcal{
\ H }}\mathfrak{D}_{t_{1}{+}}^{\kappa }x)(t)=0, & 0<t_{1}< t<t_{2},\lambda
\in \mathbb{R}^{\ast } \\ 
x(t_{1})=x(t_{2})=0, & 
\end{cases}
\label{EVP}
\end{equation}
has a nontrivial continuous solution, then  
\begin{equation}
|\lambda |\geq \frac{\Gamma (\sigma -\kappa )}{\max \left\{ \Omega ,\,\mho
\right\} }(t_{2}-t_{1}),  \label{EV}
\end{equation}
where $\Omega $ and $\mho $ are defined as in Proposition \ref{max0}.
\end{corollary}

\section{Applications}

In this section, we apply the main Lyapunov-type inequality established in
Theorem \ref{t2} to derive criteria for the nonexistence of nontrivial
solutions to the Hadamard fractional BVP (\ref{PB1}), with examples.

\begin{corollary}
\label{t3} Assume that  
\begin{equation}
\int_{t_{1}}^{t_{2}}|q(\varsigma)|\,d\varsigma<\frac{\Gamma (\sigma -\kappa
) }{\max \left\{ \Omega ,\,\mho \right\}},  \label{nonex}
\end{equation}
where $\Omega$ and $\mho$ are defined as in Proposition \ref{max0}. Then 
problem (\ref{PB1}) has no nontrivial solutions.
\end{corollary}

\bigskip

Corollary \ref{t3} yields directly the following result.

\begin{corollary}
\label{co3} Let $\Omega$ and $\mho$ be as defined in Proposition \ref{max0}
. If  
\begin{equation}
|\lambda |<\frac{\Gamma (\sigma -\kappa )}{\max \left\{ \Omega ,\,\mho
\right\} }(t_{2}-t_{1}),  \label{lmd}
\end{equation}
then the problem (\ref{EVP})  has no nontrivial solutions.
\end{corollary}

\bigskip

We now illustrate the applicability of our results through examples.

\begin{example}
\label{ex:1} Consider the Hadamard fractional BVP  
\begin{equation}
\begin{cases}
({}^{\mathcal{H}}\mathfrak{D}_{1^{+}}^{1.75}x)(t)+(^{\mathcal{H}}\mathfrak{D}
_{1^{+}}^{0.5}[\ln(\cdot)\, x(\cdot)])(t)=0, & 1<t<e, \\ 
x(1)=x(e)=0, & 
\end{cases}
\label{ex1}
\end{equation}
where $\sigma =1.75$, $\kappa =0.5$, $[t_{1},t_{2}]=[1,e]$, and $q(t)=\ln t$.

\medskip

We first compute the necessary constants:  
\begin{align*}
\Delta &=(1-\kappa )^{2}+4(\sigma -1)\left( \sigma -1-\kappa \right) \\
 &= (1-0.5 )^{2}+4(1.75 -1)\left( 1.75 -1-0.5 \right) =1.\\
X_{2}&=\frac{1}{2}\left(1+2(\sigma-1)-\kappa-\sqrt{\Delta}\right) \\
&=\frac{1}{2}[1+2(0.75)-0.5-1]=0.5.
\end{align*}

Now, we evaluate $\Omega $ and $\mho $:  
\begin{align*}
\Omega & =\frac{(X_{2})^{\sigma -1}\left( \ln \frac{t_{2}}{t_{1}}
-X_{2}\right) ^{\sigma -\kappa -1}}{\left( \ln \frac{t_{2}}{t_{1}}\right)
^{\sigma -1}\,t_{1}\,e^{X_{2}}} \\
& =\frac{(0.5)^{0.75}(1-0.5)^{0.25}}{1\times e^{0.5}}=0.3032653299, \\
\mho & =\frac{\kappa }{\sigma -1}\left( 1-\frac{\kappa }{\sigma -1}\right)
^{\frac{\sigma -\kappa -1}{\kappa }}\frac{1}{t_{1}}\left( \ln \frac{t_{2}}{
t_{1}}\right) ^{\sigma -\kappa -1} \\
& =\frac{0.5}{0.75}\left( 1-\frac{0.5}{0.75}\right) ^{0.5}\times 1\times
1^{0.25}=0.3849001795, 
\end{align*}
and $$\\
\Gamma (\sigma -\kappa )=\Gamma (1.25)=0.9064024771.$$
We obtain  
\begin{equation*}
\frac{\Gamma (\sigma -\kappa )}{\max \left\{ \Omega ,\,\mho \right\} } =
2.3549027134.
\end{equation*}

Next, computing the integral  
\begin{equation*}
\int_{1}^{e}|\ln \varsigma|\,d\varsigma=\int_{1}^{e}\ln
\varsigma\,d\varsigma=(\varsigma \ln \varsigma-\varsigma) \big| 
_{\varsigma=1}^{\varsigma=e}=1.
\end{equation*}

Since $1<2.3549027134$, the condition (\ref{nonex}) is satisfied. Therefore,
by Corollary \ref{t3}, problem (\ref{ex1}) has no nontrivial solutions.
\end{example}

\begin{example}
\label{ex:2} Consider the BVP  
\begin{equation}
\begin{cases}
({}^{\mathcal{H}}\mathfrak{D}_{1^{+}}^{1.75}x)(t)+\lambda (^{\mathcal{H}} 
\mathfrak{D}_{1^{+}}^{0.5}x)(t)=0, & 1<t<e, \\ 
x(1)=x(e)=0, & 
\end{cases}
\label{ex2}
\end{equation}
where $\lambda \in \mathbb{R}^{\ast }$.

Using the values $\sigma =1.75$, $\kappa =0.5$, and $[t_{1},t_{2}]=[1,e]$, 
the calculations from Example \ref{ex:1} yield  
\begin{equation*}
\frac{\Gamma (\sigma -\kappa )}{\max \left\{ \Omega ,\,\mho \right\} }
(e-1)=2.3549027134\times (e-1) = 4.0463865405.
\end{equation*}

By Corollary \ref{co3}, if $|\lambda |<4.0463865405$, then problem (\ref{ex2}) has no nontrivial solutions.
\end{example}


\bigskip

{\bf Conflicts of Interest.} The author declares no conflicts of interest.

\bigskip

{\bf Research funding.} This research did not receive funding.

\end{document}